\pdfoutput=1
\documentclass[11pt, reqno]{amsart}

\numberwithin{equation}{section}

\usepackage{amsmath,amsfonts,amssymb}
\usepackage{mathtools}
\usepackage[utf8]{inputenc}
\usepackage{comment}
\usepackage[english]{babel}
\usepackage{hyperref}
\usepackage{a4wide}
\usepackage{cite}
\usepackage{stmaryrd}
\usepackage[new]{old-arrows}
\usepackage[margin=3.33cm]{geometry}

\newcommand{\Sp}{\mathrm{Sp}}
\newcommand{\SO}{\mathrm{SO}}

\newcommand{\Higgs}{\mathrm{Higgs}}

\newcommand{\End}{\mathrm{End}}

\newcommand{\SPEnd}{\mathrm{SPEnd}}

\DeclareUnicodeCharacter{FB01}{fi}
\newtheorem{theorem}{Theorem}[section]

\newtheorem{lemma}[theorem]{Lemma}
\newtheorem{remark}[theorem]{Remark}
\newtheorem{example}[theorem]{Example}
\newtheorem{prop}[theorem]{Proposition}
\newtheorem*{aim-non}{Aim}

\newtheorem*{conjecture-non}{Conjecture}

\makeatletter
\def\subsection{\@startsection{subsection}{2}%
  \normalparindent{.5\linespacing\@plus.7\linespacing}{-.5em}%
  {\normalfont\bfseries}}
\def\subsubsection{\@startsection{subsubsection}{3}%
  \normalparindent\z@{-.5em}%
  {\normalfont\itshape}}
\makeatother

\theoremstyle{definition}
\newtheorem{definition}[theorem]{Definition}

\begin{document}

\title[Very stable and semiprojectivity]{Semiprojectivity and Very Stability in Moduli of Symplectic and Orthogonal Parabolic Higgs Bundles}
\author[S. Roy]{Sumit Roy}

 \address{Stat-Math Unit, Indian Statistical Institute, 203 B.T. Road, Kolkata 700 108, India.}
\email{sumitroy\_r@isical.ac.in}

\begin{abstract}
Let $X$ be a compact Riemann surface of genus $g \geq 2$, and let $D \subset X$ be a fixed finite subset. We prove the semiprojectivity of the moduli space of semistable symplectic or orthogonal parabolic Higgs bundles over $X$. We show that a stable symplectic parabolic bundle $E$ on $X$ is strongly very stable, meaning $E$ does not have any nonzero strongly parabolic nilpotent Higgs field, if and only if the symplectic parabolic Hitchin morphism induced on the affine space $$H^0(X,\mathrm{SPEnd}_\mathrm{Sp}(E) \otimes K(D))$$ is a proper morphism, where $\mathrm{SPEnd}_\mathrm{Sp}(E)$ denotes the set of symplectic strongly parabolic endomorphisms of $E$. We remark that the same criterion for very stability applies to the orthogonal case.
\end{abstract}
\maketitle

\section{Introduction}

Let $X$ be a compact Riemann surface of genus $g \geq 2$, and let $D \subset X$ be a finite subset. Parabolic bundles over $(X,D)$ were introduced by Mehta and Seshadri in \cite{MS80} in order to extend the Narasimhan--Seshadri correspondence to irreducible unitary representations of $\pi_1(X\setminus D)$. A parabolic bundle consists of a holomorphic vector bundle together with a weighted flag at each point of $D$. Bhosle and Ramanathan extended this notion to parabolic principal $G$-bundles, for $G$ a connected reductive group, and constructed the corresponding moduli spaces in \cite{BR89}. 

A symplectic (resp. orthogonal) parabolic bundle is a parabolic bundle equipped with a suitably defined nondegenerate alternating (resp. symmetric) bilinear form with values in a holomorphic line bundle; see \cite{BMW11}. When the parabolic weights are rational, these objects can be interpreted in terms of parabolic principal $G$-bundles for $G=\Sp(2m,\mathbb{C})$ (resp. $G=\SO(n,\mathbb{C})$), again in the sense of \cite{BMW11}. Thus the symplectic and orthogonal parabolic cases naturally fit into the general framework of parabolic principal bundles.

A Higgs bundle on $X$ is a pair $(E,\tau)$, where $E$ is a holomorphic vector bundle and $\tau$, called the Higgs field, is a holomorphic $1$-form with values in $\End(E)$ satisfying $\tau \wedge \tau =0$. On a compact Riemann surface, this is equivalent to a holomorphic morphism
\[
\tau:E\longrightarrow E\otimes K,
\]
where $K$ is the canonical bundle of $X$. Higgs bundles were introduced by Hitchin in \cite{H87a} in his study of the self-duality equations on a Riemann surface. A parabolic Higgs bundle on $(X,D)$ is a pair $(E_*,\Phi)$, where $E_*$ is a parabolic bundle and
\[
\Phi:E\longrightarrow E\otimes K(D)
\]
is a parabolic Higgs field. Yokogawa constructed the moduli space
\[
\mathcal{M}_{\Higgs}(n,d,\alpha)
\]
of semistable parabolic Higgs bundles of fixed rank $n$, degree $d$, and parabolic weights $\alpha$ in \cite{Y93}. The coefficients of the characteristic polynomial $\det(x\cdot I-\Phi)$ define the Hitchin morphism
\[
h:\mathcal{M}_{\Higgs}(n,d,\alpha)\longrightarrow \mathcal{A}:=\bigoplus_{i=1}^{n}H^0(X,K(D)^i),
\]
which is proper; see \cite{N91,M94}. The geometry of this morphism plays a central role in the study of Higgs bundles.

The aim of this paper is to study symplectic and orthogonal analogues of parabolic Higgs bundles from two related perspectives. The first concerns the global geometry of the corresponding moduli spaces, while the second concerns a Hitchin-theoretic characterization of very stability.

We begin with semiprojectivity. Recall that a quasi-projective variety $Z$ equipped with a $\mathbb{C}^*$-action is called \textit{semiprojective} if the fixed point locus $Z^{\mathbb{C}^*}$ is proper and, for every $z\in Z$, the limit
\[
\lim_{t\to 0} t\cdot z
\]
exists in $Z$. In the context of Higgs bundles, the natural $\mathbb{C}^*$-action is given by scaling the Higgs field. Semiprojectivity is a useful property because it provides control over the behaviour of this action and relates the geometry of the moduli space to that of the fixed point locus and the nilpotent cone. It also allows one to apply techniques such as Bia{\l}ynicki--Birula type decompositions in the study of the topology of these spaces. Semiprojectivity for moduli spaces of parabolic Higgs bundles was established in \cite{R24}. Our first result extends this to the symplectic and orthogonal settings.

\begin{theorem}
The moduli spaces
\[
\mathcal{M}^{\Sp}_{\Higgs}(\alpha),\qquad
\mathcal{M}^{\SO(2m)}_{\Higgs}(\alpha),\qquad
\mathcal{M}^{\SO(2m+1)}_{\Higgs}(\alpha),
\]
and their strongly parabolic analogues are semiprojective.
\end{theorem}

The second part of the paper is motivated by the notion of very stability. In the non-parabolic case, Pauly and Pe\'on-Nieto proved in \cite{PP19} that a stable vector bundle is very stable if and only if the Hitchin morphism associated to it is proper. Pe\'on-Nieto extended this result to the parabolic setting in \cite{P24}, and Zelaci obtained an analogous criterion for principal $G$-bundles in \cite{Z20}. Our goal is to obtain a corresponding result for symplectic parabolic bundles.

Let $(E_*,\varphi)$ be a symplectic parabolic bundle. We denote by
\[
\mathcal{W}_{E,\mathrm{st}}
=
H^0\bigl(X,\SPEnd_{\Sp}(E_*)\otimes K(D)\bigr)
\]
the vector space of strongly parabolic symplectic Higgs fields on $(E_*,\varphi)$. If $(E_*,\varphi)$ is stable, then each element of $\mathcal{W}_{E,\mathrm{st}}$ determines a stable strongly parabolic symplectic Higgs bundle, and hence defines a morphism
\[
h_{E,\mathrm{st}}:\mathcal{W}_{E,\mathrm{st}}\longrightarrow \mathcal{A}^{\mathrm{st}}_{\Sp}
\]
by composition with the symplectic Hitchin morphism. We say that $(E_*,\varphi)$ is \textit{strongly very stable} if it admits no nonzero nilpotent strongly parabolic symplectic Higgs field. The following theorem provides a characterization of this condition.

\begin{theorem}
Let $(E_*,\varphi)$ be a stable symplectic parabolic bundle. Then the following are equivalent:
\begin{align*}
(E_*,\varphi)\ \textnormal{is strongly very stable}
&\iff h_{E,\mathrm{st}}\ \textnormal{is finite}\\
&\iff h_{E,\mathrm{st}}\ \textnormal{is proper}\\
&\iff h_{E,\mathrm{st}}\ \textnormal{is quasi-finite}.
\end{align*}
\end{theorem}

Thus, in the symplectic parabolic setting, strong very stability is completely determined by the geometry of the associated Hitchin morphism. We also indicate the corresponding statement in the orthogonal case (see Remark \ref{orthogonal}).

The proofs of semiprojectivity rely on the $\mathbb{C}^*$-equivariance of the Hitchin morphism together with its properness. The very-stability criterion is obtained by combining the homogeneity of the Hitchin map with the linear structure of the space $\mathcal{W}_{E,\mathrm{st}}$.

The results of this paper suggest several directions for further study. One may investigate the topology and cohomology of these moduli spaces using their semiprojective structure, or study the geometry of the corresponding Hitchin fibres in the symplectic and orthogonal parabolic settings. It would also be natural to extend these results to more general parabolic principal $G$-Higgs bundles.

\section{Preliminaries}
\subsection{Parabolic bundles}
	Let us choose a set $D$ of $r$ distinct points $p_1, ..., p_r$ on $X$.We will keep this set $D$ fixed throughout this paper.
	\begin{definition}\label{parabolic}
		 A \textit{parabolic bundle} $E_*$ of rank $n$ over the curve $X$ is a rank $n$ holomorphic vector bundle $E$ over $X$ with an additional structure (call the \textit{parabolic structure}) along $D$, i.e. for every $p\in D$, we have
	\begin{enumerate}
		\item a descending filtration of subspaces of the fiber over $p$
		\[
		E_p \eqqcolon E_{p,1}\supsetneq E_{p,2} \supsetneq \dots \supsetneq E_{p,n_p} \supsetneq E_{p,n_p+1} =\{0\},
		\]
		\item an ascending sequence of real numbers between $0$ and $1$ satisfying 
		\[
		0\leq \alpha_1(p) < \alpha_2(p) < \dots < \alpha_{n_p}(p) < 1,
		\]
	\end{enumerate}
where $n_p$ is an integer between $1$ and $n$. 
\end{definition}

Here, the real numbers $\alpha_i(p)$'s are called the \textit{parabolic weights} assigned to the subspace $E_{p, i}$, for all $p\in D$. A \textit{quasi-parabolic structure} on $E$ consists of flags (without the parabolic weights) over all points of $D$. For a fixed parabolic structure, we call the set of all parabolic weights by $$\alpha =\{(\alpha_1(p),\alpha_2(p),\dots,\alpha_{n_p}(p))\}_{p\in D}.$$ The parabolic structure $\alpha$ is said to have \textit{full flags} if, for all $i$ and for all $p \in D$:
\[
\mathrm{dim}(E_{p,i}/E_{p,i+1}) = 1,
\] 
or equivalently $n_p=n$ for all $p\in D$.

In \cite{MY92}, Maruyama and Yokogawa provided a different way to define parabolic bundles using coherent sheaves. This approach helps define parabolic tensor products and parabolic duals.

\begin{definition}
     The \textit{parabolic degree} of a parabolic bundle $E_*$ is defined by
\[
\operatorname{pardeg}(E_*) \coloneqq \deg(E)+ \sum\limits_{p\in D}\sum\limits_{i=1}^{n_p} \alpha_i(p) \cdot \dim(E_{p,i}/E_{p,i+1}),
\]
where $\deg(E)$ is the degree of the underlying vector bundle $E$, and the \textit{parabolic slope} of $E_*$ is defined by
\[
\mu_{\mathrm{par}}(E_*) \coloneqq \frac{\text{pardeg}(E_*)}{\mathrm{rank}(E)}.
\]
\end{definition}

\begin{definition}
	A \textit{parabolic subbundle} $F_*$ of a parabolic bundle $E_*$ is a holomorphic subbundle $F\subset E$ of the underlying bundle $E$, which also has a parabolic structure induced on it. This induced parabolic structure on $F$ is defined as follows: for every $p\in D$, the filtration in $F_p$ is given by 
	\[
		F_p \eqqcolon F_{p,1}\supsetneq F_{p,2} \supsetneq \dots \supsetneq F_{p,n'_p} \supsetneq \{0\},
	\]
	where $F_{p,i}= F_p \cap E_{p,i}$, i.e. we are looking at the intersection with the filtration already given in $E_p$, and we are also removing any repeated subspaces in the sequence. The parabolic weights $0\leq \alpha'_1(p) < \alpha'_2(p) < \dots < \alpha'_{n'_p}(p) < 1$ are chosen to be the largest possible weights among those allowed after the intersections, i.e. 
	\[
	\alpha'_i(p) = \mathrm{max}_j\{\alpha_j(p)| F_p \cap E_{p,j}=F_{p,i} \}=\mathrm{max}_j\{\alpha_j(p)| F_{p,i}\subseteq E_{p,j} \}
	\]
	In other words, the parabolic weight associated to the subspace $F_{p,i}$ is the weight $\alpha_j(p)$ such that $F_{p,i}\subseteq E_{p,j}$ but $F_{p,i}\nsubseteq E_{p,j+1}$.
\end{definition}
\begin{definition}
	A parabolic bundle $E_*$ is said to be \textit{stable} (resp. \textit{semistable}) if for all nonzero proper parabolic subbundle $F_* \subset E_*$, we have
	\[
	\mu_{\mathrm{par}}(F_*) < \mu_{\mathrm{par}}(E_*) \hspace{0.2cm} (\mathrm{resp. } \hspace{0.2cm} \leq).
	\]
\end{definition}
\begin{definition}
A \textit{parabolic homomorphism} $\phi: E_* \to E^\prime_*$ between two parabolic bundles is a vector bundle homomorphism between the underlying vector bundles which satisfies 
\[
\alpha_i(p) > \alpha_j^\prime(p) \implies \phi(E_{p,i}) \subseteq E_{p,j+1}^\prime,
\]
at each $p \in D$. Moreover, we say a homomorphism is \textit{strongly parabolic} if, for every $p \in D$, we have
\[ \alpha_i(p) \geq \alpha_j^\prime(p) \implies \phi(E_{p,i}) \subseteq E_{p,j+1}^\prime.
\]
\end{definition}

We will use $\mathrm{PEnd}(E_*)$ and $\mathrm{SPEnd}(E_*)$ to denote the parabolic and strongly parabolic endomorphisms of $E_*$, respectively.
\subsection{Parabolic Higgs bundles}
	Let $K$ be the canonical line bundle on $X$. We define $K(D) \coloneqq K \otimes \mathcal{O}(D)$ with the trivial parabolic structure (i.e. weights are zero). 
\begin{definition}
A \textit{parabolic Higgs bundle} on $X$ consists of a parabolic bundle $E_*$ and a parabolic homomorphism $$\Phi: E_* \to E_* \otimes K(D),$$ i.e. for every $p \in D$ we have $\Phi(E_{p,i}) \subset E_{p,i} \otimes \left.K(D)\right|_p$. This homomorphism $\Phi$ is called the \textit{parabolic Higgs field}.
 
A \textit{strongly parabolic Higgs bundle} on $X$ consists of a parabolic bundle $E_*$ together with a \textit{strongly parabolic Higgs field}
\[
\Phi : E_* \longrightarrow E_* \otimes K(D),
\]
such that for every $p \in D$,
\[
\Phi(E_{p,i}) \subset E_{p,i+1} \otimes K(D)|_p.
\]

Equivalently, the residue of $\Phi$ at each marked point $p \in D$ is nilpotent with respect to the induced filtration on the fiber $E_p$. Indeed, the above condition implies that the induced endomorphism
\[
\operatorname{Res}_p(\Phi) : E_p \longrightarrow E_p
\]
strictly lowers the filtration, and hence is nilpotent.
\end{definition}

For a fixed parabolic bundle $E_*$, the set of strongly parabolic Higgs fields on $E_*$ is given by
\[
H^0\bigl(X,\mathrm{SPEnd}(E_*)\otimes K(D)\bigr),
\]
which is a vector subspace of
\[
H^0\bigl(X,\mathrm{PEnd}(E_*)\otimes K(D)\bigr).
\]
Indeed, the condition
\[
\Phi(E_{p,i}) \subset E_{p,i+1}\otimes K(D)|_p,\qquad p\in D,
\]
is linear in $\Phi$. In particular, the strongly parabolic condition defines a closed condition in families, and hence the strongly parabolic locus is closed in the moduli space of parabolic Higgs bundles.

\begin{definition}
    A subbundle $F_* \subset E_*$ of a parabolic Higgs bundle $(E_*,\Phi)$ is called \textit{$\Phi$-invariant} if $\Phi(F_*) \subset F_* \otimes K(D)$.
\end{definition}

\begin{definition}
	A parabolic Higgs bundle $(E_*,\Phi)$ is said to be \textit{stable} (resp. \textit{semistable}) if for every nonzero proper $\Phi$-invariant subbundle $F_* \subset E_*$, we have
	\[
	\mu_{\mathrm{par}}(F_*) < \mu_{\mathrm{par}}(E_*) \hspace{0.2cm} (\mathrm{resp. } \hspace{0.2cm} \leq).
	\]
\end{definition}
\begin{lemma}\label{induce}
	Let $E_*$ and $E'_*$ be parabolic bundles together with the (strongly) parabolic Higgs fields $\Phi_E$ and $\Phi_{E'}$ respectively. Then there is a (strongly) parabolic Higgs field on the tensor product $E_* \otimes E'_*$ induced by $\Phi_E$ and $\Phi_{E'}$. Also, $\Phi_E$ induces a (strongly) parabolic Higgs field on the dual $E_*^\vee$.
\end{lemma}
\begin{proof}
	See \cite[Lemma 4]{BS12}.
\end{proof}	

\subsection{Moduli of parabolic Higgs bundles and Hitchin morphism}
 In \cite{K93}, Konno constructed the moduli space $\mathcal{M}_{\mathrm{Higgs}}(n,d,\alpha)$ of semistable parabolic Higgs bundles of fixed rank $n$, degree $d$, and parabolic structure $\alpha$ over $X$ (see also \cite{Y93}, \cite{BY96} for more details). It is a normal quasi-projective variety. The stable locus of $\mathcal{M}_{\mathrm{Higgs}}(n,d,\alpha)$ is an open subvariety of $\mathcal{M}_{\mathrm{Higgs}}(n,d,\alpha)$ and is the smooth locus. The parabolic form of Serre duality (see \cite{BY96}, \cite{Y95}) states that 
\begin{equation}\label{parSerre}
H^1(X,\mathrm{PEnd}(E_*)) \cong H^0(X,\mathrm{SPEnd}(E_*) \otimes K(D))^*.
\end{equation}
Hence, by deformation theory, there exists an open embedding $$T^*\mathcal{M}(n,d,\alpha) \xhookrightarrow{}  \mathcal{M}^{\mathrm{st}}_{\mathrm{Higgs}}(n,d,\alpha),$$ where $T^*\mathcal{M}(n,d,\alpha)$ is the cotangent bundle of the moduli space $\mathcal{M}(n,d,\alpha)$  of parabolic bundles (see \cite{MS80}), and $\mathcal{M}^{\mathrm{st}}_{\mathrm{Higgs}}(n,d,\alpha) \subset \mathcal{M}_{\mathrm{Higgs}}(n,d,\alpha)$ is the locus of strongly parabolic Higgs bundles.

\subsubsection{Parabolic Hitchin morphism}
Let $\mathcal{S}$ denote the total space of $K(D)$. Let $p: \mathcal{S} \to X $ be the natural projection map, and let $x \in H^0(\mathcal{S}, p^*K(D))$ be the tautological section. Let $(E_*,\Phi)$ be a semistable parabolic Higgs bundle of rank $n$, and let
\[
\det(x\,\mathrm{Id} - p^*\Phi) = x^n + \tilde{s}_1 x^{n-1} + \cdots + \tilde{s}_n
\]
be the characteristic polynomial of $\Phi$, where $\tilde{s}_i = p^* s_i$ for some
\[
s_i \in H^0\bigl(X, K^i(D^i)\bigr).
\]
Here,
\[
K^i(D^j) := K^{\otimes i} \otimes \mathcal{O}(D)^{\otimes j}.
\]

Sometimes we will write $K^i(D^i)$ as $K(D)^i$. Therefore, the parabolic \textit{Hitchin morphism} is defined by
\begin{align*}
h :\mathcal{M}_{\mathrm{Higgs}}(n,d,\alpha) &\longrightarrow \label{key}\mathcal{A} \coloneqq \bigoplus_{i=1}^n H^0(X, K^i(D^{i}))  \\
(E_*,\Phi) &\longmapsto (s_1,\dots,s_n),
\end{align*}
i.e. $h$ maps a parabolic Higgs field to the coefficients of its characteristic polynomial. The Hitchin morphism $h$ is independent of the parabolic structure at each $p\in D$ because it only relies on the Higgs field $\Phi$ and $K(D)$. Moreover, $h$ is a proper morphism (see \cite{M94} for details). 

If $\Phi$ is a strongly parabolic Higgs field, then its residue at each marked point $p \in D$ is nilpotent, and hence $s_i \in H^0(X, K^i(D^{i-1}))$. Therefore the restriction of the Hitchin morphism to the strongly parabolic locus is given by
\[
h_{\mathrm{st}} \coloneqq \left.h\right|_{\mathcal{M}^{\mathrm{st}}_{\mathrm{Higgs}}(n,d,\alpha)} :\mathcal{M}^{\mathrm{st}}_{\mathrm{Higgs}}(n,d,\alpha) \longrightarrow \label{key}\mathcal{A}^{\mathrm{st}} \coloneqq \bigoplus_{i=1}^n H^0(X, K^i(D^{i-1})).
\]
The fiber $h^{-1}(0)$ is called the \textit{nilpotent cone} of the Hitchin morphism $h$. One can see that $h^{-1}(0)$ consists of all nilpotent Higgs fields. 

\subsection{Symplectic and orthogonal parabolic bundles}

By \cite{BMW11}, for rational parabolic weights, symplectic and orthogonal parabolic bundles can be interpreted in terms of parabolic principal $G$-bundles, where $G$ is the corresponding symplectic or orthogonal group. We will use this identification throughout.

\medskip

\noindent\textbf{Symplectic case.}
Let us consider the standard symplectic form $J$ on $\mathbb{C}^{2m}$ given by
\[
J = \begin{bmatrix} 0 & I_{m} \\ -I_{m} & 0 \end{bmatrix}.
\]
The general symplectic group $\mathrm{Gp}(2m,\mathbb{C})$ is defined by
\[
\mathrm{Gp}(2m,\mathbb{C}) = \{ M \in \mathrm{GL}(2m,\mathbb{C}) : MJM^t=\mu_M J \text{ for some } \mu_M \in \mathbb{C}^*\}.
\]
It is an extension of $\mathbb{C}^*$ by $\mathrm{Sp}(2m,\mathbb{C})$, that is, there exists a short exact sequence
\[
1 \to \mathrm{Sp}(2m,\mathbb{C}) \longrightarrow \mathrm{Gp}(2m,\mathbb{C}) \xrightarrow[]{p} \mathbb{C}^* \to 1,
\]
where $p(M)=\mu_M$. In particular,
\[
\det(M)=p(M)^m=\mu_M^m
\]
for all $M \in \mathrm{Gp}(2m,\mathbb{C})$.

The Lie algebra of $\mathrm{Gp}(2m,\mathbb{C})$ is given by
\[
\mathfrak{gp}(2m,\mathbb{C}) = \{A \in \mathfrak{gl}(2m,\mathbb{C}) : AJ + JA^t = \tfrac{\mathrm{tr}(A)}{m}J\},
\]
and admits a decomposition
\[
\mathfrak{gp}(2m,\mathbb{C}) \cong \mathfrak{sp}(2m,\mathbb{C}) \oplus \mathbb{C}.
\]

\medskip

\noindent\textbf{Orthogonal case.}
Let $B$ be a symmetric non-degenerate bilinear form on $\mathbb{C}^{n}$, and denote by the same symbol the corresponding matrix. The general orthogonal group $\mathrm{GO}(n,\mathbb{C})$ is defined by
\begin{equation}\label{eqn1}
\mathrm{GO}(n,\mathbb{C}) = \{ M \in \mathrm{GL}(n,\mathbb{C}) : MBM^t=\mu_M B \text{ for some } \mu_M \in \mathbb{C}^*\}.
\end{equation}

There is a short exact sequence
\[
1 \to \mathrm{O}(n,\mathbb{C}) \longrightarrow \mathrm{GO}(n,\mathbb{C}) \xrightarrow[]{q} \mathbb{C}^* \to 1,
\]
where $q(M)=\mu_M$. In particular,
\[
\det(M)^2=q(M)^n.
\]
The subgroup $\mathrm{SO}(n,\mathbb{C}) \subset \mathrm{O}(n,\mathbb{C})$ consists of elements with determinant $1$.

\medskip

\noindent\textbf{Symplectic and orthogonal parabolic bundles.}
Let $L_*$ be a parabolic line bundle over $X$, and let $E_*$ be a parabolic bundle. Let
\[
\varphi : E_* \otimes E_* \longrightarrow L_*
\]
be a parabolic homomorphism. By tensoring with the dual parabolic bundle, we obtain an induced parabolic homomorphism
\begin{equation}\label{nondegenerate}
\tilde{\varphi} : E_* \longrightarrow L_* \otimes E^\vee_*.
\end{equation}

\begin{definition}\label{maindefn}
A \textit{symplectic parabolic bundle} consists of a pair $(E_*,\varphi)$ as above, where $\varphi$ is anti-symmetric and $\tilde{\varphi}$ is an isomorphism. 
	
An \textit{orthogonal parabolic bundle} consists of a pair $(E_*,\varphi)$ where $\varphi$ is symmetric and $\tilde{\varphi}$ is an isomorphism.
\end{definition}
 Given a symplectic or orthogonal parabolic bundle $(E_*,\varphi)$, let $E$ be its underlying vector bundle. Then $E \otimes E$ is a coherent subsheaf of the vector bundle underlying $E_*\otimes E_*$. Thus, $\varphi$ induces a vector bundle homomorphism
\begin{equation}\label{underlying}
\hat{\varphi} : E \otimes E \to L,
\end{equation}
where $L$ is the underlying line bundle of $L_*$. 

\begin{definition}
A subbundle $F \subset E$ of the underlying vector bundle $E$ of $E_*$ is said to be \textit{isotropic} if $\hat{\varphi}(F \otimes F)=0$.
\end{definition}

\begin{definition}\label{compatible}
Let $(E_*,\varphi)$ be a symplectic or orthogonal parabolic bundle on $X$. Thus, $\varphi : E_* \otimes E_* \to L_*$ is a bilinear form inducing an isomorphism
\[
\tilde{\varphi} : E_* \longrightarrow L_* \otimes E_*^\vee.
\]

Let 
\[
\Phi \in H^0\bigl(X,\mathrm{PEnd}(E_*)\otimes K(D)\bigr)
\]
be a parabolic Higgs field. By Lemma \ref{induce}, $\Phi$ induces a Higgs field $\Phi^\vee$ on $E_*^\vee$, and hence a Higgs field on $L_* \otimes E_*^\vee$. We say that $\Phi$ is \textit{compatible} with $\varphi$ if
\[
(\mathrm{Id}_{L_*}\otimes \Phi^\vee)\circ \tilde{\varphi}
=
(\tilde{\varphi}\otimes \mathrm{Id}_{K(D)})\circ \Phi.
\]
\end{definition}

\medskip

\begin{remark}
Let $(E_*,\varphi)$ be a symplectic or orthogonal parabolic bundle. The bilinear form $\varphi$ induces an isomorphism
\[
\widetilde{\varphi} : E_* \longrightarrow L_* \otimes E_*^\vee.
\]
A parabolic Higgs field $\Phi$ is said to be compatible with $\varphi$ if it corresponds, under this identification, to the induced Higgs field on $L_* \otimes E_*^\vee$.

Equivalently, at the level of fibers, this condition means that $\Phi$ is skew-symmetric (in the symplectic case) or symmetric (in the orthogonal case) with respect to the bilinear form $\varphi$ (for example see \cite{GGR09} for reference).
\end{remark}

\begin{definition}[Symplectic or orthogonal parabolic Higgs bundles]
A \textit{symplectic} (resp. \textit{orthogonal}) \textit{parabolic Higgs bundle} $(E_*,\varphi,\Phi)$ consists of a symplectic (resp. orthogonal) parabolic bundle $(E_*,\varphi)$ and a Higgs field $\Phi$ on $E_*$ that is compatible with $\varphi$.
\end{definition}

\begin{definition}
	A symplectic or orthogonal parabolic Higgs bundle $(E_*,\varphi,\Phi)$ is called \textit{stable} (resp. \textit{semistable}) if for all nonzero isotropic $\Phi$-invariant subbundle $F \subset E$ (i.e. $F_*$ is $\Phi$-invariant), we have
	\[
	\mu_{par}(F_*) < \mu_{par}(E_*) \hspace{0.4cm}(\text{resp.} \hspace{0.15cm} \mu_{par}(F_*) \leq \mu_{par}(E_*)).
	\]
\end{definition}

\subsubsection{Moduli space}
A description of the moduli space $\mathcal{M}^G(\alpha)$ of semistable parabolic $G$-bundles with fixed parabolic structure $\alpha$ and rank $r$ and degree $d$ was given in \cite{BBN01} and \cite{BR89}. They showed that it is a projective variety. Moreover, if the parabolic structure $\alpha$ has full flags at each marked points $p\in D$, then its dimension is given by
\[
\dim\mathcal{M}^G(\alpha)= \dim Z(G) +  (g-1)\dim(G) + r\dim(G/B),
\]
where $r$ is the number of points in $D$ and $Z(G)$ is the center of $G$. The final term arises because the flags considered at each point of $D$ are full flags, and $B$ represents the Borel subgroup of $G$ specified by $\alpha$. We will assume that the parabolic structures at all points of $D$ are full flags.

If $G=\mathrm{Sp}(2m,\mathbb{C})$ or $\mathrm{SO}(n,\mathbb{C})$, then
\[
\dim\mathcal{M}^G(\alpha) = (g-1)\dim(G) + r\dim(G/B).
\]
Moreover, the moduli space $\mathcal{M}^G_{\mathrm{Higgs}}(\alpha)$ of semistable parabolic $G$-Higgs bundles is a quasi-projective variety (see \cite{Y93}) and by the parabolic version of Serre duality, we know that
\[
\dim \mathcal{M}^{G,\mathrm{st}}_{\mathrm{Higgs}}(\alpha) = 2\dim \mathcal{M}^G(\alpha),
\]
where $\mathcal{M}^{G,\mathrm{st}}_{\mathrm{Higgs}}(\alpha) \subset \mathcal{M}_{\mathrm{Higgs}}(\alpha)$ is the strongly parabolic locus. 

For $G=\mathrm{Sp}(2m,\mathbb{C})$ and $\mathrm{SO}(2m+1,\mathbb{C})$, the respective moduli spaces $\mathcal{M}^{\mathrm{Sp},\mathrm{st}}_{\mathrm{Higgs}}(\alpha)$ and $\mathcal{M}^{\mathrm{SO}(2m+1),\mathrm{st}}_{\mathrm{Higgs}}(\alpha)$ have equal dimension, and it is given by
\[
\dim \mathcal{M}^{\mathrm{Sp},\mathrm{st}}_{\mathrm{Higgs}}(\alpha) =\dim \mathcal{M}^{\mathrm{SO}(2m+1),\mathrm{st}}_{\mathrm{Higgs}}(\alpha)= 2m(2m+1)(g-1) + 2m^2r.
\]
For $G=\mathrm{SO}(2m,\mathbb{C})$, the moduli space $\mathcal{M}^{\mathrm{SO}(2m),\mathrm{st}}_{\mathrm{Higgs}}(\alpha)$ has dimension
\[
\dim \mathcal{M}^{\mathrm{SO}(2m),\mathrm{st}}_{\mathrm{Higgs}}(\alpha) = 2m(2m-1)(g-1) + 2mr(m-1).
\]

We now use these structural properties of the moduli space to study the behaviour of the Hitchin morphism under the natural $\mathbb{C}^*$-action.

\section{Semiprojectivity}\label{semiproj}
In this section, we prove that the moduli spaces of semistable symplectic or orthogonal parabolic Higgs bundles over a smooth projective curve are semiprojective varieties.

	\begin{definition}[Semiprojective variety]\label{semiprojective}
	Let $Z$ be a quasi-projective variety with a $\mathbb{C}^*$-action $z \mapsto t\cdot z$, where $t \in \mathbb{C}^*$ and $z\in Z$. We say that $Z$ is a \textit{semiprojective} variety if the following conditions hold:
	\begin{enumerate}
	    \item the limit $\lim_{t\to 0} (t\cdot z)_{t\in \mathbb{C}^*}$ exists in $Z$ for all $z \in Z$.
	    \vspace{0.1cm}
	    \item the fixed point set $Z^{\mathbb{C}^*}$ of the $\mathbb{C}^*$-action on $Z$ is a proper variety.
\end{enumerate}
\end{definition}

\subsection{Semiprojectivity of the moduli space of semistable symplectic parabolic Higgs bundles}
An element $(E_*,\varphi,\Phi)$ of the moduli space $\mathcal{M}^{\mathrm{Sp}}_{\mathrm{Higgs}}(\alpha)$ can be interpreted as a rank $2m$ parabolic bundle $E_*$ equipped with a non-degenerate symplectic form $\langle \cdot , \cdot \rangle$ (i.e., the map $\varphi$) along with a holomorphic section $\Phi \in H^0(X, \mathrm{PEnd}(E_*)\otimes K(D))$ that satisfies
\[
\langle \Phi v, w \rangle = - \langle v, \Phi w \rangle.
\]
This identity is equivalent to the compatibility condition in Definition \ref{compatible}. Therefore, the characteristic polynomial of $\Phi$ is of the form
\[
\det(x\cdot I - \Phi) = x^{2m} + s_2x^{2m-2} + \cdots + s_{2m},
\]
where $\{s_2,\dots, s_{2m}\}$ is a basis for the invariant polynomials of the Lie algebra $\mathfrak{sp}(2m)$. Thus the Hitchin morphism for the symplectic case is given by
\begin{align}
\begin{split}
    h_{\mathrm{Sp}} : \mathcal{M}^{\mathrm{Sp}}_{\mathrm{Higgs}}(\alpha) &\rightarrow \mathcal{A}_{\mathrm{Sp}}\coloneqq \bigoplus_{i=1}^{m}H^0(X, K(D)^{2i})\\
    \Phi &\longmapsto (s_2,\dots,s_{2m}),
    \end{split}
\end{align}
(see \cite{R20} for more details) and its restriction to the strongly parabolic locus is given by
\begin{align}
    h_{\mathrm{Sp,st}} \coloneqq \left.h_{\mathrm{Sp}}\right|_{\mathcal{M}^{\mathrm{Sp,st}}_{\mathrm{Higgs}}(\alpha)}: \mathcal{M}^{\mathrm{Sp,st}}_{\mathrm{Higgs}}(\alpha) \rightarrow \mathcal{A}^{\mathrm{st}}_{\mathrm{Sp}}\coloneqq \bigoplus_{i=1}^{m}H^0(X, K^{2i}(D^{2i-1})).
\end{align}

There is a natural $\mathbb{C}^*$-action on the bases $\mathcal{A}_{\mathrm{Sp}}$ and $\mathcal{A}^{\mathrm{st}}_{\mathrm{Sp}}$ of the symplectic Hitchin morphisms, given by
\[
t\cdot (s_2,\dots , s_{2m}) \coloneqq (t^2s_2,\dots , t^{2m}s_{2m}).
\]
\begin{lemma}\label{equivariant}
    The Hitchin morphism $h_{\mathrm{Sp}} : \mathcal{M}^{\mathrm{Sp}}_{\mathrm{Higgs}}(\alpha) \rightarrow \mathcal{A}_{\mathrm{Sp}}$ is $\mathbb{C}^*$-equivariant.
\end{lemma}
\begin{proof}
    Let $h_{\mathrm{Sp}}((E_*,\varphi, \Phi)) = (s_2,\dots, s_{2m})$, where $s_i = \mathrm{tr}(\wedge^i\Phi)$. Then the characteristic polynomial corresponding to $t\Phi$ is 
    \[
\det(x\cdot I - t\Phi) = x^{2m} + t^2s_2x^{2m-2} + \cdots + t^{2m}s_{2m}.
\]
Thus, 
\begin{align*}
   h_{\mathrm{Sp}}\big((t\cdot (E_*,\varphi,\Phi))\big) &=h_{\mathrm{Sp}}\big((E_*,\varphi,t\Phi)\big)\\ 
   &=(t^2s_2,\dots , t^{2m}s_{2m}) \\
   &= t\cdot (s_2,\dots , s_{2m}) \\
   &= t\cdot h_{\mathrm{Sp}}\big((E_*,\varphi,\Phi)\big). 
\end{align*}

Hence, $h_{\mathrm{Sp}}$ is $\mathbb{C}^*$-equivariant.
\end{proof}

We now study the behaviour of the $\mathbb{C}^*$-action on the moduli space and show that limits exist as $t \to 0$.
\begin{lemma}\label{limit1}
Let $(E_*,\varphi,\Phi)$ be a semistable symplectic parabolic Higgs bundle. Then $\lim_{t\to 0} (E_*,\varphi,t\Phi)$ exists in $\mathcal{M}^{\mathrm{Sp}}_{\mathrm{Higgs}}(\alpha)$.
\end{lemma}
\begin{proof}
Consider the morphism 
\begin{align*}
    f: \mathbb{C}^* &\longrightarrow \mathcal{M}^{\mathrm{Sp}}_{\mathrm{Higgs}}(\alpha)\\
    t &\longmapsto (E_*,\varphi,t\Phi).
\end{align*}
Since the Hitchin map $h_{\mathrm{Sp}}$ is $\mathbb{C}^*$-equivariant by Lemma \ref{equivariant}, we have
\[
\lim_{t\to 0} h_{\mathrm{Sp}}\big((E_*,\varphi,t\Phi)\big) = \lim_{t \to 0} t \cdot h_{\mathrm{Sp}}\big((E_*,\varphi,\Phi)\big) = 0.
\]
Therefore, the composition morphism 
$$F \coloneqq h_{\mathrm{Sp}} \circ f : \mathbb{C}^* \longrightarrow \mathcal{A}_{\mathrm{Sp}}$$ 
extends to a morphism 
$$\hat{F} : \mathbb{C} \longrightarrow \mathcal{A}_{\mathrm{Sp}},$$ 
since the $\mathbb{C}^*$-action on the Hitchin base is algebraic, given by polynomial scaling, and the above limit shows that $\hat F(0)=0$. Since the Hitchin map $h_{\mathrm{Sp}}$ is proper, it satisfies the valuative criterion of properness. The morphism $f : \mathbb{C}^* \to \mathcal{M}^{\mathrm{Sp}}_{\mathrm{Higgs}}(\alpha)$ corresponds to a morphism from $\operatorname{Spec}(\mathbb{C}((t)))$, and the extension $\hat{F}$ gives a morphism from $\operatorname{Spec}(\mathbb{C}[[t]])$ to $\mathcal{A}_{\mathrm{Sp}}$. Therefore, by the valuative criterion, $f$ extends uniquely to a morphism
\[
\hat{f}: \mathbb{C} \longrightarrow \mathcal{M}^{\mathrm{Sp}}_{\mathrm{Higgs}}(\alpha).
\] Therefore, $\lim_{t\to 0} (E_*,\varphi,t\Phi)$ exists in $\mathcal{M}^{\mathrm{Sp}}_{\mathrm{Higgs}}(\alpha)$.
\end{proof}

We next analyze the fixed point locus of the $\mathbb{C}^*$-action and relate it to the Hitchin fiber over $0$.

\begin{lemma}\label{fixed1}
The fixed point set under the $\mathbb{C}^*$-action on the moduli space $\mathcal{M}^{\mathrm{Sp}}_{\mathrm{Higgs}}(\alpha)$ is proper in $h_{\mathrm{Sp}}^{-1}(0) \subset \mathcal{M}^{\mathrm{Sp}}_{\mathrm{Higgs}}(\alpha)$.
\end{lemma}
\begin{proof}
Note that the zero point in the Hitchin base $\mathcal{A}_{\mathrm{Sp}}$ is the only point that is fixed by the $\mathbb{C}^*$-action. Thus, $\mathcal{A}_{\mathrm{Sp}}^{\mathbb{C}^*} = \{0\}$.

Since the Hitchin morphism $h_{\mathrm{Sp}}$ is $\mathbb{C}^*$-equivariant by Lemma \ref{equivariant}, we have
\[
\mathcal{M}^{\mathrm{Sp}}_{\mathrm{Higgs}}(\alpha)^{\mathbb{C}^*}
\subset h_{\mathrm{Sp}}^{-1}(\mathcal{A}_{\mathrm{Sp}}^{\mathbb{C}^*})
= h_{\mathrm{Sp}}^{-1}(0).
\]

Moreover, the fixed point locus of the $\mathbb{C}^*$-action is closed, since it is defined by algebraic conditions. Hence,
\[
\mathcal{M}^{\mathrm{Sp}}_{\mathrm{Higgs}}(\alpha)^{\mathbb{C}^*}
\subset h_{\mathrm{Sp}}^{-1}(0)
\]
is a closed subset.

Since $h_{\mathrm{Sp}}$ is proper, the fiber $h_{\mathrm{Sp}}^{-1}(0)$ is proper. Therefore, being a closed subset of a proper variety,
\[
\mathcal{M}^{\mathrm{Sp}}_{\mathrm{Higgs}}(\alpha)^{\mathbb{C}^*}
\]
is proper.
\end{proof}
\begin{theorem}\label{symplectic}
The moduli space $\mathcal{M}^{\mathrm{Sp}}_{\mathrm{Higgs}}(\alpha)$ of semistable symplectic parabolic Higgs bundles is a semiprojective variety.
\end{theorem}
\begin{proof}
The moduli space $\mathcal{M}^{\mathrm{Sp}}_{\mathrm{Higgs}}(\alpha)$ is a quasi-projective variety. Hence, the semiprojectivity follows from Lemma \ref{limit1} and \ref{fixed1}.
\end{proof}
By the same argument, we obtain
\begin{theorem}
    The moduli space $\mathcal{M}^{\mathrm{Sp},\mathrm{st}}_{\mathrm{Higgs}}(\alpha)$ of semistable symplectic strongly parabolic Higgs bundles is a semiprojective variety.
\end{theorem}

The same strategy applies to the orthogonal case, with the corresponding Hitchin bases determined by the invariant polynomials of $\mathfrak{so}(n)$.

\subsection{Semiprojectivity of the moduli space of semistable orthogonal parabolic Higgs bundles}
A point $(E_*,\varphi,\Phi)$ of the moduli space $\mathcal{M}^{\mathrm{SO}(2m)}_{\mathrm{Higgs}}(\alpha)$ (resp. $\mathcal{M}^{\mathrm{SO}(2m+1)}_{\mathrm{Higgs}}(\alpha)$) can be interpreted as a rank $2m$ (resp. rank $2m+1$) parabolic bundle $E_*$ equipped with a non-degenerate symmetric bilinear form $\langle \cdot , \cdot \rangle$ (i.e., the map $\varphi$) along with a holomorphic section $\Phi \in H^0(X, \mathrm{PEnd}(E_*)\otimes K(D))$ that satisfies
\[
\langle \Phi v, w \rangle =  \langle v, \Phi w \rangle.
\]
\subsubsection{Even orthogonal case $\mathrm{(}G=\mathrm{SO}(2m,\mathbb{C})\mathrm{)}$}

\vspace{0.01cm}

In the orthogonal case, similar to the symplectic case, the characteristic polynomial of a Higgs field $\Phi$ is given by
\[
\det(x\cdot I - \Phi) = x^{2m} + s_2x^{2m-2} + \cdots + s_{2m-2}x^2+ s_{2m},
\]
where $s_i = \mathrm{trace}(\wedge^i \Phi)$.  In this case, the last coefficient $s_{2m}=p_m^2$, where $p_m \in H^0(X,K(D)^m)$ is called the Pfaffian. Thus $\{s_2,\dots,s_{2m-2},p_m\}$ is a basis for the invariant polynomials of the Lie algebra $\mathfrak{so}(2m)$. Therefore, the even orthogonal Hitchin morphism is given by
\begin{align}
\begin{split}
      h_{\mathrm{SO}(2m)} : \mathcal{M}^{\mathrm{SO}(2m)}_{\mathrm{Higgs}}(\alpha) &\rightarrow \mathcal{A}_{\mathrm{SO}(2m)} \coloneqq \bigoplus_{i=1}^{m-1} H^0(X, K(D)^{2i}) \oplus H^0(X, K(D)^m)\\
    \Phi &\mapsto (s_2, \dots, s_{2m-2},p_m),
    \end{split}
\end{align}
and the restriction $ h_{\mathrm{SO}(2m), \mathrm{st}} \coloneqq \left.h_{\mathrm{SO}(2m)}\right|_{\mathcal{M}^{\mathrm{SO}(2m),\mathrm{st}}_{\mathrm{Higgs}}(\alpha)}$ of the Hitchin map to the strongly parabolic locus is given by
\begin{align}
    h_{\mathrm{SO}(2m), \mathrm{st}}: \mathcal{M}^{\mathrm{SO}(2m),\mathrm{st}}_{\mathrm{Higgs}}(\alpha) \rightarrow \mathcal{A}^{\mathrm{st}}_{\mathrm{SO}(2m)} \coloneqq \bigoplus_{i=1}^{m-1} H^0(X, K^{2i}(D^{2i-1})) \oplus H^0(X, K(D)^m).
\end{align}

There is a natural $\mathbb{C}^*$-action on the bases $\mathcal{A}_{\mathrm{SO}(2m)}$ and $\mathcal{A}^{\mathrm{st}}_{\mathrm{SO}(2m)}$ given by 
\[
t\cdot (s_2,\dots, s_{2m-2},p_m) \coloneqq (t^2s_2, \dots, t^{2m-2}s_{2m-2}, t^mp_m).
\]

The next result follows from the homogeneity of the characteristic polynomial under scaling of the Higgs field; see, for example, \cite{H87a, GGR09}.

\begin{lemma}\label{equivariant1}
    The Hitchin morphism $h_{\mathrm{SO}(2m)} : \mathcal{M}^{\mathrm{SO}(2m)}_{\mathrm{Higgs}}(\alpha) \rightarrow \mathcal{A}_{\mathrm{SO}(2m)}$ is $\mathbb{C}^*$-equivariant.
\end{lemma}
\begin{proof}
Let $(E_*,\varphi,\Phi)\in \mathcal{M}^{\mathrm{SO}(2m)}_{\mathrm{Higgs}}(\alpha)$ and let $$h_{\mathrm{SO}(2m)}\big((E_*,\varphi,\Phi)\big) = (s_2,\dots, s_{2m-2},p_m^2).$$ Then the characteristic polynomial of $t\Phi$ is
    \[
\det(x\cdot I - t\Phi) = x^{2m} + t^2s_2x^{2m-2} + \cdots +t^{2m-2}s_{2m-2}+ t^{2m}p_m^2.
\]
Thus, 
\begin{align*}
   h_{\mathrm{SO}(2m)}\big((t\cdot (E_*,\varphi,\Phi))\big) &=h_{\mathrm{SO}(2m)}\big((E_*,\varphi,t\Phi)\big)\\ 
   &=(t^2s_2,\dots ,t^{2m-2}s_{2m-2}, t^{m}p_m) \\
   &= t\cdot (s_2,\dots ,s_{2m-2}, p_m) \\
   &= t\cdot h_{\mathrm{SO}(2m)}\big((E_*,\varphi,\Phi)\big). 
\end{align*}

Hence, $h_{\mathrm{SO}(2m)}$ is $\mathbb{C}^*$-equivariant.
\end{proof}

\begin{theorem}\label{even-orthogonal}
    The moduli space $\mathcal{M}^{\mathrm{SO}(2m)}_{\mathrm{Higgs}}(\alpha)$ of semistable even orthogonal parabolic Higgs bundles is a semiprojective variety.
\end{theorem}
\begin{proof}
    Since $\mathcal{M}^{\mathrm{SO}(2m)}_{\mathrm{Higgs}}(\alpha)$ is a quasi-projective variety and the Hitchin map $h_{\mathrm{SO}(2m)}$ is $\mathbb{C}^*$-equivariant by Lemma \ref{equivariant1} and proper, we can conclude the semiprojectivity of $\mathcal{M}^{\mathrm{SO}(2m)}_{\mathrm{Higgs}}(\alpha)$ by using a method similar to the one used in Lemma \ref{limit1} and \ref{fixed1}.
\end{proof}
Similarly, we get
\begin{theorem}
    The moduli space $\mathcal{M}^{\mathrm{SO}(2m),\mathrm{st}}_{\mathrm{Higgs}}(\alpha)$ of semistable even orthogonal strongly parabolic Higgs bundles is a semiprojective variety.
\end{theorem}

\subsubsection{Odd orthogonal case $\mathrm{(}G=\mathrm{SO}(2m+1,\mathbb{C})\mathrm{)}$}
Assuming the Higgs field $\Phi$ has distinct eigenvalues, its characteristic polynomial in this case is given by
\[
\det(x\cdot I - \Phi) = x(x^{2m} + s_2x^{2m-2} + \cdots + s_{2m-2}x^2+ s_{2m}).
\]
Thus, the Hitchin morphism for the odd orthogonal group is similar to the symplectic case, and it is given by
\begin{align}
\begin{split}
    h_{\mathrm{SO}(2m+1)} : \mathcal{M}^{\mathrm{SO}(2m+1)}_{\mathrm{Higgs}}(\alpha) &\rightarrow \mathcal{A}_{\mathrm{SO}(2m+1)} \coloneqq \oplus_{i=1}^{m} H^0(X, K(D)^{2i})\\
    \varphi &\mapsto (s_2, \dots, s_{2m-2},s_{2m}),
    \end{split}
\end{align}
and the restriction $ h_{\mathrm{SO}(2m+1), \mathrm{st}} \coloneqq \left.h_{\mathrm{SO}(2m+1)}\right|_{\mathcal{M}^{\mathrm{SO}(2m+1),\mathrm{st}}_{\mathrm{Higgs}}(\alpha)}$ of the Hitchin map to the strongly parabolic locus is given by
\begin{align}
    h_{\mathrm{SO}(2m+1), \mathrm{st}}: \mathcal{M}^{\mathrm{SO}(2m+1),\mathrm{st}}_{\mathrm{Higgs}}(\alpha) \rightarrow \mathcal{A}^{\mathrm{st}}_{\mathrm{SO}(2m+1)} \coloneqq \bigoplus_{i=1}^{m} H^0(X, K^{2i}(D^{2i-1})).
\end{align}

The natural $\mathbb{C}^*$-action on the bases $\mathcal{A}_{\mathrm{SO}(2m+1)}$ and $\mathcal{A}^{\mathrm{st}}_{\mathrm{SO}(2m+1)}$ are given by 
\[
t\cdot (s_2,\dots, s_{2m}) \coloneqq (t^2s_2, \dots, t^{2m}s_{2m}).
\]

\begin{lemma}\label{equivariant2}
    The Hitchin morphism $h_{\mathrm{SO}(2m+1)} : \mathcal{M}^{\mathrm{SO}(2m+1)}_{\mathrm{Higgs}}(\alpha) \rightarrow \mathcal{A}_{\mathrm{SO}(2m+1)}$ is $\mathbb{C}^*$-equivariant.
\end{lemma}
\begin{proof}
Let $(E_*,\varphi,\Phi)\in \mathcal{M}^{\mathrm{SO}(2m+1)}_{\mathrm{Higgs}}(\alpha)$ and let $$h_{\mathrm{SO}(2m+1)}\big((E_*,\varphi,\Phi)\big) = (s_2,\dots, s_{2m-2},s_{2m}).$$
  Then the characteristic polynomial of $t\Phi$ is 
    \[
\det(x\cdot I - t\Phi) = x(x^{2m} + t^2s_2x^{2m-2} + \cdots + t^{2m}s_{2m}).
\]
Thus, 
\begin{align*}
   h_{\mathrm{SO}(2m+1)}\big((t\cdot (E_*,\varphi,\Phi))\big) &=h_{\mathrm{SO}(2m+1)}\big((E_*,\varphi,t\Phi)\big)\\ 
   &=(t^2s_2,\dots , t^{2m}s_{2m}) \\
   &= t\cdot (s_2,\dots ,s_{2m}) \\
   &= t\cdot h_{\mathrm{SO}(2m+1)}\big((E_*,\varphi,\Phi)\big). 
\end{align*}

Hence, $h_{\mathrm{SO}(2m+1)}$ is $\mathbb{C}^*$-equivariant.
\end{proof}

Therefore, we get the following theorem
\begin{theorem}\label{odd-orthogonal}
    The moduli space $\mathcal{M}^{\mathrm{SO}(2m+1)}_{\mathrm{Higgs}}(\alpha)$ of semistable odd orthogonal parabolic Higgs bundles is a semiprojective variety.
\end{theorem}
\begin{proof}
    The proof is identical to that of Theorem \ref{even-orthogonal}, using the existence of limits of the $\mathbb{C}^*$-action and the closedness of the fixed point locus. Therefore, $\mathcal{M}^{\mathrm{SO}(2m+1)}_{\mathrm{Higgs}}(\alpha)$ is semiprojective.
\end{proof}

Similarly, we get
\begin{theorem}
        The moduli space $\mathcal{M}^{\mathrm{SO}(2m+1),\mathrm{st}}_{\mathrm{Higgs}}(\alpha)$ of semistable odd orthogonal strongly parabolic Higgs bundles is a semiprojective variety.
\end{theorem}

\section{Very stable symplectic parabolic bundles}

In this section, we establish a criterion for the (strong) very stability of a symplectic parabolic bundle in terms of the Hitchin morphism.

Let $(E_*,\varphi)$ be a symplectic parabolic bundle, where
\[
\varphi : E_* \otimes E_* \longrightarrow L_*
\]
is a nondegenerate alternating parabolic form with values in a parabolic line bundle $L_*$. The symmetric power $\mathrm{Sym}^2(E_*)$ is understood with the induced parabolic structure coming from the tensor product $E_* \otimes E_*$; see, for example, \cite{MY92,Y93}. Define
\[
\mathrm{PEnd}_{\mathrm{Sp}}(E_*) \coloneqq \mathrm{Sym}^2(E_*) \otimes L_*^\vee \subset \mathrm{PEnd}(E_*)=E_* \otimes E_* \otimes L_*^\vee,
\]
and similarly let
\[
\mathrm{SPEnd}_{\mathrm{Sp}}(E_*) \subset \mathrm{SPEnd}(E_*)
\]
denote the sheaf of symmetric symplectic strongly parabolic endomorphisms of $E_*$.

Set
\[
\mathcal{W}_E \coloneqq H^0\bigl(X,\mathrm{PEnd}_{\mathrm{Sp}}(E_*) \otimes K(D)\bigr),
\]
and
\[
\mathcal{W}_{E,\mathrm{st}} \coloneqq H^0\bigl(X,\mathrm{SPEnd}_{\mathrm{Sp}}(E_*) \otimes K(D)\bigr)\subset \mathcal{W}_E.
\]

If $(E_*,\varphi)$ is stable, then every $\Phi \in \mathcal{W}_E$ defines a stable symplectic parabolic Higgs bundle $(E_*,\varphi,\Phi)$. Hence we obtain a natural morphism
\[
\iota_E : \mathcal{W}_E \longrightarrow \mathcal{M}^{\Sp}_{\Higgs}(\alpha),
\qquad
\Phi \longmapsto (E_*,\varphi,\Phi).
\]
Similarly, every $\Phi \in \mathcal{W}_{E,\mathrm{st}}$ defines a stable strongly parabolic symplectic Higgs bundle, and therefore we obtain a natural morphism
\[
\iota_{E,\mathrm{st}} : \mathcal{W}_{E,\mathrm{st}} \longrightarrow \mathcal{M}^{\Sp,\mathrm{st}}_{\Higgs}(\alpha).
\]

We now define
\begin{equation}\label{eq:hE}
h_E \coloneqq h_{\Sp}\circ \iota_E : \mathcal{W}_E \longrightarrow \mathcal{A}_{\Sp},
\end{equation}
and
\begin{equation}\label{eq:hEst}
h_{E,\mathrm{st}} \coloneqq h_{\Sp,\mathrm{st}}\circ \iota_{E,\mathrm{st}} : \mathcal{W}_{E,\mathrm{st}} \longrightarrow \mathcal{A}^{\mathrm{st}}_{\Sp}.
\end{equation}

\begin{definition}
A symplectic parabolic bundle $(E_*,\varphi)$ is said to be \textit{very stable} if it admits no nonzero nilpotent symplectic parabolic Higgs field in $\mathcal{W}_E$. It is said to be \textit{strongly very stable} if it admits no nonzero nilpotent strongly parabolic symplectic Higgs field in $\mathcal{W}_{E,\mathrm{st}}$.
\end{definition}

\begin{remark}
Every very stable symplectic parabolic bundle is strongly very stable, since
\[
\mathcal{W}_{E,\mathrm{st}} \subset \mathcal{W}_E.
\]
\end{remark}

\begin{lemma}\label{lemma:svs-stable}
If a symplectic parabolic bundle $(E_*,\varphi)$ is strongly very stable, then it is stable. Consequently, every very stable symplectic parabolic bundle is stable.
\end{lemma}
\begin{proof}
The proof is the same as in \cite[Lemma $3.5$]{P24} (see also \cite[Proposition $3.5$]{L88}).
\end{proof}

We will use the following lemma for our purpose.
\begin{lemma}{\cite[Lemma $1.3$]{Z20}}\label{finite}
Let $f=(f_1,\dots,f_k):\mathbb{A}^m \longrightarrow \mathbb{A}^k$ be a morphism defined by homogeneous polynomials. If $f^{-1}(0)=\{0\}$, then $f$ is finite.
\end{lemma}

\begin{theorem}\label{quasi-finite}
Let $(E_*,\varphi)$ be a stable symplectic parabolic bundle. Then
\[
h_E \textnormal{ is quasi-finite } \implies (E_*,\varphi) \textnormal{ is very stable } \implies h_E \textnormal{ is finite}.
\]
\end{theorem}
\begin{proof}
Assume first that $h_E$ is quasi-finite. Suppose that $(E_*,\varphi)$ is not very stable. Then there exists a nonzero nilpotent Higgs field
\[
0 \neq \Phi \in \mathcal{W}_E
\]
such that $h_E(\Phi)=0$. Since the Hitchin morphism is given by homogeneous invariant polynomials, for every $t \in \mathbb{C}^*$ we have
\[
h_E(t\Phi)=t\cdot h_E(\Phi)=0.
\]
Hence
\[
\mathbb{C}^* \cdot \Phi \subset h_E^{-1}(0).
\]
Since $\mathcal{W}_E$ is a vector space and $\Phi \neq 0$, the set $\mathbb{C}^*\cdot \Phi$ is infinite, contradicting the quasi-finiteness of $h_E$. Therefore $(E_*,\varphi)$ is very stable.

Conversely, assume that $(E_*,\varphi)$ is very stable. Then the only nilpotent element in $\mathcal{W}_E$ is $0$, so
\[
h_E^{-1}(0)=\{0\}.
\]
Since $h_E$ is defined by homogeneous polynomials on the affine space $\mathcal{W}_E$, Lemma \ref{finite} implies that $h_E$ is finite.
\end{proof}

We now focus on the strongly parabolic locus.

\begin{prop}\label{proper2}
Let $(E_*,\varphi)$ be a stable symplectic parabolic bundle. Then
\[
h_{E,\mathrm{st}} \textnormal{ is finite } \iff h_{E,\mathrm{st}} \textnormal{ is proper}.
\]
In particular, if $h_{E,\mathrm{st}}$ is finite, then it is quasi-finite.
\end{prop}

\begin{proof}
Suppose $h_{E,\mathrm{st}}$ is finite. Then it is proper and quasi-finite.

Now assume that $h_{E,\mathrm{st}}$ is proper. Since both $\mathcal{W}_{E,\mathrm{st}}$ and $\mathcal{A}^{\mathrm{st}}_{\Sp}$ are affine varieties, the morphism $h_{E,\mathrm{st}}$ is affine. A proper affine morphism is finite. Hence $h_{E,\mathrm{st}}$ is finite.
\end{proof}

\begin{prop}\label{stronglyverystable}
Let $(E_*,\varphi)$ be a stable symplectic parabolic bundle. Then
\[
h_{E,\mathrm{st}} \textnormal{ is quasi-finite } \implies (E_*,\varphi) \textnormal{ is strongly very stable } \implies h_{E,\mathrm{st}} \textnormal{ is finite}.
\]
\end{prop}

\begin{proof}
The proof is the same as in Theorem \ref{quasi-finite}, replacing $\mathcal{W}_E$ by $\mathcal{W}_{E,\mathrm{st}}$ and $h_E$ by $h_{E,\mathrm{st}}$.
\end{proof}

\begin{theorem}\label{maintheorem}
Let $(E_*,\varphi)$ be a stable symplectic parabolic bundle. Then
\begin{align*}
(E_*,\varphi) \textnormal{ is strongly very stable }
&\iff h_{E,\mathrm{st}} \textnormal{ is finite} \\
&\iff h_{E,\mathrm{st}} \textnormal{ is proper} \\
&\iff h_{E,\mathrm{st}} \textnormal{ is quasi-finite}.
\end{align*}
\end{theorem}

\begin{proof}
If $(E_*,\varphi)$ is strongly very stable, then by Proposition \ref{stronglyverystable}, the morphism $h_{E,\mathrm{st}}$ is finite. By Proposition \ref{proper2}, this is equivalent to properness, and in particular implies quasi-finiteness.

Conversely, if $h_{E,\mathrm{st}}$ is quasi-finite, then by Proposition \ref{stronglyverystable}, $(E_*,\varphi)$ is strongly very stable. This completes the proof.
\end{proof}

\begin{example}
Consider the case $G=\Sp(2,\mathbb{C}) \cong \mathrm{SL}(2,\mathbb{C})$. In this case, a symplectic parabolic bundle reduces to a rank $2$ parabolic bundle with trivial determinant equipped with a nondegenerate alternating form.

The space $\mathcal{W}_{E,\mathrm{st}}$ consists of strongly parabolic traceless endomorphisms twisted by $K(D)$, and the Hitchin map is given by the determinant, which lies in $H^0(X,K(D)^2)$. In this setting, strongly very stability means that there are no nonzero nilpotent Higgs fields, i.e., no Higgs fields with vanishing determinant.

Thus, Theorem \ref{maintheorem} reduces in this case to the statement that the Hitchin map on $\mathcal{W}_{E,\mathrm{st}}$ is proper if and only if there are no nonzero nilpotent Higgs fields.
\end{example}

\begin{remark}\label{orthogonal}
Using the same method, one obtains an analogous characterization of strong very stability in terms of the corresponding Hitchin map in the orthogonal case, for stable locus of the moduli space $\mathcal{M}^{G,\mathrm{st}}_{\mathrm{Higgs}}(\alpha)$, where $G=\mathrm{SO}(2m,\mathbb{C})$ or $\mathrm{SO}(2m+1,\mathbb{C})$. We omit the details, as the arguments are entirely parallel to the symplectic case.
\end{remark}

\section*{Acknowledgment}
 The author is supported by the INSPIRE faculty fellowship (Ref No.: IFA22-MA 186) funded by the DST, Govt. of India.

\end{document}